\documentclass[11pt]{amsart}

\usepackage[english]{babel}
\usepackage[T1]{fontenc}
\usepackage{tgschola}
\usepackage[utf8]{inputenc} 

\usepackage{amssymb,latexsym}
\usepackage{mathtools}
\usepackage{enumerate}  
\usepackage{a4wide}    
\usepackage{color}

\numberwithin{equation}{section}

\theoremstyle{plain}
\newtheorem{lemma}{\bf Lemma}[section]
\newtheorem{theorem}[lemma]{\bf Theorem}

\newtheorem{problem}[lemma]{\bf Problem}           
\newtheorem{fact}[lemma]{\bf Proposition}  
           
\theoremstyle{definition}
\newtheorem{definition}[lemma]{\bf Definition}

\theoremstyle{remark}
\newtheorem{remark}[lemma]{\bf Remark}
\newtheorem*{remark*}{\bf Remark}

\newcommand{\R}{\mathbb{R}}

\newcommand{\Sphere}{\mathbb{S}}

\newcommand{\eps}{\varepsilon}

\newcommand{\Rz}{\mathcal{R}}
\def\Jac{\mathbf{J}}

\newcommand{\dv}{\mathrm{div}\,}
\newcommand{\Ep}{\bigwedge\nolimits}
\newcommand{\lap}{\Delta\,}

\newcommand{\laps}[1]{{(-\lap)^{\frac{#1}{2}}}}

\newcommand{\aleq}{\precsim}

\def\charfn_#1{{\raise1.2pt\hbox{$\chi  %
_{\kern-1pt\lower4pt\hbox{{$\scriptstyle#1$}}}$}}}  %

\def\mvint_#1{\mathchoice
	          {\mathop{\vrule width 6pt height 3 pt depth -2.5pt
	            \kern -9pt \intop}\nolimits_{\kern -3pt #1}}%
{\mathop{\vrule width 5pt height 3 pt depth -2.6pt
	                        \kern -6pt \intop}\nolimits_{#1}}%
{\mathop{\vrule width 5pt height 3 pt depth -2.6pt
	                   \kern -6pt \intop}\nolimits_{#1}}%
{\mathop{\vrule width 5pt height 3 pt depth -2.6pt
	                      \kern -6pt \intop}\nolimits_{#1}}}
                                      
\newcommand{\omitted}[1]{}

\subjclass[2010]{Primary 35B65, 35J60, 35J70, secondary 58E20.}

\renewcommand{\MR}[1]{}   

\begin{document}
\renewcommand{\baselinestretch}{1.03}

\title[Regularity for $H$-systems\,{}?]{Invitation to $H$-systems in higher dimensions:\\ known results, new facts, and related open problems}

\author[A. Schikorra]{Armin Schikorra}
\address{Armin Schikorra \newline \indent
Mathematisches Institut, 
Albert-Ludwigs-Universit\"at Freiburg\newline \indent
Eckerstra\ss{}e 1,
79104 Freiburg im Breisgau,
Germany}
\email{armin.schikorra@math.uni-freiburg.de}

\author[P. Strzelecki]{Pawe\l{} Strzelecki}
\address{Pawe\l{} Strzelecki\newline \indent
Institute of Mathematics, University of Warsaw\newline \indent
Banacha 2, 02--097 Warszawa, Poland}
\email{P.Strzelecki@mimuw.edu.pl}

\date{Version of \today}

\begin{abstract} In this paper, we discuss two well-known open problems in the regularity theory for nonlinear, conformally invariant elliptic systems in dimensions $n\ge 3$, with a critical nonlinearity: $H$-systems (equations of hypersurfaces of prescribed
mean curvature) and $n$-har\-mo\-nic maps into compact Riemannian manifolds. 

For $n=2$ several solutions of these problems are known but they all break down in higher dimensions (unless one considers special cases, e.g. hypersurfaces of constant mean curvature or manifolds with symmetries). We discuss some of the known proofs and hint at the main difficulties.

We also state a few new results (such as positive answers for all solutions of class $W^{n/2,2}$ for even $n$, instead of $W^{1,n}$) and list some open questions of independent interest --- including specific endpoint variants of the Coifman-Rochberg-Weiss theorem, addressing the boundedness of commutators of fractional and singular integrals with multiplication by bounded functions of class $W^{1,n}$ --- that would lead to solutions of these two problems.

\end{abstract}

\maketitle

\setcounter{tocdepth}{1}
\tableofcontents

\section{Introduction}

25 years ago F. H\'{e}lein \cite{Helein-1991} proved that all harmonic maps from planar domains into compact Riemannian manifolds are smooth. Besides H\'{e}lein's own insight into the method of using the moving frames to rewrite the right-hand side of the harmonic map equation and reveal its divergence structure, the main analytical ingredient of this achievement was the discovery that certain nonlinear expressions -- like the Jacobian of a map $u\in W^{1,n}(\R^n,\R^n)$, or various `div-curl' quantities -- enjoy, due to \emph{cancellation\/} phenomena, slightly better regularity or integrability properties than those that would follow only from their growth properties.   

H\'{e}lein's ideas and the powerful Hardy space methods based on \cite{Muller-1990,CLMS} have triggered a stream of research. In particular, F. Bethuel \cite{Bethuel-1992} proved that all weak solutions $u\in W^{1,2}(B^2,\R^3)$  of the equation of surfaces in $\R^3$ that have prescribed mean curvature, $\Delta u = 2H(u) u_x\times u_y$, where $H\colon \R^3\to \R$ is bounded and Lipschitz, are continuous. 

Fifteen years later, in an influential paper \cite{Riviere-2007}, T. Rivi\`{e}re has derived a general conservation law for solutions $u\in W^{1,2}(B^2,\R^m)$ of 
\begin{equation}\label{rivieresystem} \Delta u = \Omega \cdot \nabla u ,\end{equation} 
where $\Omega$ is an $L^2$-matrix with values in $so(m)\otimes\R^2$. An important point of \cite{Riviere-2007} was that the antisymmetry of $\Omega$ can be used to replace the divergence structure; due to this, all weak solutions of  $\Delta u = \Omega \cdot \nabla u$ are continuous. To an untrained eye, the result looks dry and technical, but two long-standing open problems were its corollaries: 
a conjecture by S. Hildebrandt claiming that critical points of 
elliptic conformally invariant Lagrangians in two dimensions are continuous (this was known before only under a boundedness assumption on the map or a stronger smoothness assumption on the target manifold, due to P. Chon\'{e} \cite{Chone-1995}),
and a conjecture by E. Heinz asserting that the solutions to the prescribed mean curvature equation 
with only \emph{bounded} mean curvature
are continuous.

Some of the 
generalizations and extensions of H\'{e}lein's work  -- we discuss them in more detail below
-- were concerned with  applications of the ideas and techniques of \cite{CLMS} in dimensions $n\geq 3$, to conformally invariant problems, in particular to equations involving the $n$-Laplace operator. One of them is
\begin{equation}\label{nriviere-intro}
-\dv \bigl(|\nabla u|^{n-2}\nabla u\bigr) = \Omega \cdot |\nabla u|^{n-2} \nabla u, \qquad u \in W^{1,n}(B^n,\R^m)\,  \end{equation}
with an antisymmetric $\Omega$ of class $L^{n}$, cf. \cite[eq. III.23]{Riviere-2011} and Problem~\ref{rivierepb} below. Despite the efforts and interest of numerous authors, three natural $n$-dimensional counterparts of the results of \cite{Helein-1991}, \cite{Bethuel-1992} and \cite{Riviere-2007}, in particular the regularity questions for $n$-harmonic maps and for $H$-systems in dimension $n>2$, are still open; only partial results are known. For $n=2$, some of the milestones described above were due to successful applications of \emph{linear} harmonic analysis to \emph{nonlinear} problems; it seems that for $n>2$ a deeper understanding of the nonlinearity (and maybe of the underlying geometry) is simply missing.

It is our aim to describe these problems and some attempts at their solutions. There is a lot of circumstantial evidence -- in the form of partial positive results in simplified cases or under extra assumptions that do not trivialize the problems -- that the answers might be positive, and the counterexamples, if they exist at all, would have to be rather subtle.

This paper is, for the most part, a survey. However, later on, we do insert 
a few 
new results which provide part of the evidence 
alluded to above. These new observations include 
\begin{itemize}
\item regularity of bounded solutions to $H$-systems with the mean curvature $H$ being just bounded and H\"{o}lder continuous;
\item in even dimensions, regularity of all $W^{n/2,2}$--solutions 
to $H$-systems with $H$ bounded and Lipschitz;
\item a splitting result for $H$-systems, allowing to write the right side as a sum of the determinant terms and a term $\Omega_{ij}\cdot |\nabla u|^{n-2} \nabla u^j$ with an antisymmetric matrix $\Omega$ in the Lorentz space $L^{(n,n/2)}\subsetneq L^n$;
\item a higher integrability result for a toy non-local version of \eqref{nriviere-intro}.
\end{itemize}
However, our main wish is to  attract the attention of the reader (a) to those specific regularity problems involving the $n$--Laplace operator, (b) to some related problems of analysis that in our opinion are interesting in their own right.

\section{Statement of the problems}       
          
\def\nnn{{\mathcal N}}

\subsection{Regularity of $n$-harmonic maps into compact Riemannian manifolds.} Let $\mathcal N$ be a compact closed Riemannian
manifold, isometrically 
embedded in $\R^m$, with $\pi\colon \R^m\supset B_\delta(\nnn)\to \nnn$ 
being the standard nearest point projection of a tubular  
neighbourhood of $\nnn$ onto $\nnn$. 
Let $\Omega
\subset \R^n$ be open and   
bounded, and let $p\in (1,n]$. Consider 
mappings $u:\Omega\to\nnn$ such that the $p$-Dirichlet energy of 
$u$, given by the functional 
\begin{equation} 
\label{p-energy} 
E_p[u]\colon= \frac 1p \int_\Omega |\nabla u|^p\, dx 
= \frac 1p \int_\Omega \Biggl(\sum_{i,j}
  \biggl(\frac{\partial u^i}{\partial x_j}\biggr)^2 
  \Biggr)^{p/2}\, dx\, , 
\end{equation} 
is finite. Here, $u=(u^1,\ldots,u^m) \colon \Omega \to \R^m$ 
is a mapping with all coordinates $u^j\in W^{1,p}(\Omega)$,
satisfying the extra constraint $u(x)\in\nnn$ for a.e. $x\in\Omega$. 
The class of all such maps is traditionally denoted by 
$W^{1,p}(\Omega, \nnn)$. 
 
\begin{definition} 
A map
$u\in W^{1,p}(\Omega, \nnn)$ is 
{\em (weakly) $p$-harmonic\/} if and only if $u$ is a critical point 
of $E_p$ with respect  to variations in the range, i.e. 
\begin{equation} 
\label{d-over-dt} 
 \biggl.\frac{d}{dt}\biggr|_{t=0}  
        E_p\Bigl[\pi\circ ( u+t\varphi)\Bigr]  =  0 
\qquad\mbox{for each 
$\varphi\in W^{1,p}_0(\Omega,{\R}^m)\cap  
            L^\infty(\Omega,\R^m).$ }
\end{equation} 
\end{definition}  
For $p=2$ one simply says that $u$ is a harmonic map into $\nnn$.
 
A computation (see e.g. M.~Fuchs \cite{fuchs-book} 
or H\'{e}lein \cite{helein-book}) shows 
that (\ref{d-over-dt}) yields the Euler--Lagrange 
system 
\begin{equation}
-\dv(|\nabla u|^{p-2}\nabla u )\perp T_u\, \nnn 
\qquad\mbox{in the sense of ${\mathcal D}'(\Omega,\R^n)$,} 
\label{div-perp}
\end{equation}        
or, equivalently,
\begin{equation}     
	\label{phar}
   -\dv(|\nabla u|^{p-2}\nabla u )=|\nabla u|^{p-2}A_u(\nabla u, \nabla u) ,
\end{equation}
where $A$ denotes the second fundamental form of the embedding $\nnn\subset \R^d$.   

For an excellent review of numerous issues concerning (partial) regularity  of harmonic and $p$-harmonic maps, we refer to R.~Hardt's survey \cite{Hardt-1997} (the case of $p\not =2$ is reviewed briefly in \cite[Sec.~8]{Hardt-1997}; we give more references below).

In the case $p < n$, one cannot expect \emph{any} regularity for \emph{weakly} $p$-harmonic maps without an extra assumption. Already for $p = 2$ and $n > 2$, Rivi\`{e}re \cite{Riviere-1995} showed the existence of \emph{everywhere} discontinuous harmonic maps into the round sphere. If one considers \emph{minimizing} $p$-harmonic maps one has sharp partial regularity results: such maps are regular outside a closed singular set which has Hausdorff dimension at most $n- \lfloor p \rfloor -1$, see C.B. Morrey \cite{Morrey-1948} and R.~Schoen and K.~Uhlenbeck \cite{Schoen-Uhlenbeck-1982} for $p = 2$, 
and for $p \geq 2$ R.~Hardt and F.H.~Lin \cite{Hardt-Lin-1987}, Fuchs \cite{Fuchs-1989,Fuchs-1990} and S.~Luckhaus \cite{Luckhaus-1988}. L.~Simon \cite{Simon-1995} proved that the singular set of a minimizing harmonic map is rectifiable. For \emph{stationary} $p$-harmonic maps into symmetric targets, the singular set satisfies ${\mathcal{H}}^{n-p} (\text{sing}\, u)=0$, see Fuchs \cite{Fuchs-1993}, Takeuchi \cite{Takeuchi-1994}, Toro and Wang \cite{ToroW-1995}, the second author's \cite{Strzelecki-1994,Strzelecki-1996}, and the recent work by A.~Naber, D.~Valtorta and G.~Veronelli \cite{Naber-2015}.\footnote{For $p=2$, Naber and Valtorta \cite{Naber2-2015} prove new results on the rectifiability of the strata of the singular set of a \emph{stationary} harmonic map.} For generalizations we also refer to \cite{Hajlasz-Strzelecki-1998,Wang-2003,Strzelecki-2003b}.

Despite several results that we describe in the next section, the following problem is still open.
\begin{problem}[$p=n$] 
\label{n-harmonic}
Are all weakly $n$-harmonic maps $u\in W^{1,n}(B^n,\nnn)$ continuous?
\end{problem}

\subsection{$H$-systems.} The $H$-systems are closely related to 
$n$-harmonic maps, and from the viewpoint of regularity theory they present 
the same analytical difficulties in an a bit simpler framework. They come up as
the Euler--Lagrange systems for the $n$-Dirichlet energy plus a volume term; 
one of the sources of motivation is that conformal solutions of an $H$-system  
parametrize hypersurfaces of prescribed mean curvature.

Let $u=(u^1,\ldots,u^{n+1})\colon B^n \to \R^{n+1}$ be of class $W^{1,n}$. 
Set
\[
\Jac u = \frac{\partial u}{\partial x_1} \times \cdots \times \frac{\partial u}{\partial x_n} \, ;
\]                   
for $x\in B^n$ this is a vector in $\R^{n+1}$ with coordinates given by the $n\times n$ minors of the Jacobi matrix of $Du$.

\begin{definition}  For a bounded function $H\colon \R^{n+1}\to\R$, an $H$-system is 
	\begin{equation}   
		\label{hsystem}
		-\dv \bigl(|\nabla u|^{n-2}\nabla u\bigr) = H(u)\, \Jac u, \qquad u\in W^{1,n}(B^n,\R^{n+1}).
\end{equation}       

For \emph{constant} $H$, solutions of \eqref{hsystem} are also known as $n$-harmonic maps with prescribed volume. Namely, for a map $u=(u^1,\ldots,u^{n+1})\colon B^n \to \R^{n+1}$ of class $W^{1,n}$ one can define the volume of the cone over $u(B^n)$ with vertex at $0\in \R^{n+1}$ as
\[
V(u)=\frac 1{n+1}\int_{B^n} u\cdot \Jac u\, dx\, .
\]
For the minimization problem
\[
\min_u \int_{B^n} |\nabla u|^n\, dx 
\]
under a prescribed Dirichlet boundary condition for $u$ on $\partial B^n$ \emph{and} prescribed volume $V(u)=\mathrm{const}$, \eqref{hsystem} is the Euler--Lagrange system and the constant $H$ is just the Euler--Lagrange multiplier. For variable $H$, a variational approach to the existence of solutions of \eqref{hsystem} is set forth by F.~Duzaar and J.~Grotowski in \cite{DuzaarG-2000}.

\begin{problem}\label{hsystempb} Let $n>2$. Suppose that $H\colon \R^{n+1}\to\R$ is bounded and Lipschitz. Are all weak solutions $u\in W^{1,n}(B^n,\R^{n+1})$ of \eqref{hsystem} continuous?
\end{problem}

For $n=2$ the answer is positive, see Bethuel~\cite{Bethuel-1992}. We discuss 
known partial evidence for $n>2$, including a few new observations, in the next section.           
In dimension $n=2$, Problems~\ref{n-harmonic} and~\ref{hsystempb} are closely related: the same tools of mathematical analysis (Hardy space and BMO duality, or its variants) can be used to prove regularity of solutions. Basically, H\'{e}lein's Coulomb moving frame allows one to rewrite the equation of harmonic maps in a form analogous to \eqref{hsystem}, and the case of symmetric target manifolds corresponds to $H$ being constant. For $n>2$, Wang \cite{Wang-2005}, in his paper on weak limits of $n$-harmonic maps, gives a construction of an appropriate moving frame, see also \cite{Miskiewicz-2016}. Thus, we are tempted to think that a solution to Problem~\ref{hsystempb} would open the way to Problem~\ref{n-harmonic}.
	
\end{definition}                        
Rivi\`{e}re in his survey article \cite[eq. III.23]{Riviere-2011} poses the following problem on regularity of degenerate systems with an \emph{antisymmetric potential}. A positive answer to his question would imply a positive answer to Problem~\ref{hsystempb} with $H$ just of class $L^\infty$, and to Problem~\ref{n-harmonic} for all $C^2$-manifolds.

\begin{problem}\label{rivierepb}
Let $u \in W^{1,n}(B^n,\R^{m})$ satisfy a system of the form
    \begin{equation}\label{nrivieresystem}
-\dv \bigl(|\nabla u|^{n-2}\nabla u\bigr) = \Omega \cdot |\nabla u|^{n-2} \nabla u. \end{equation} 
Assume that $\Omega = (\Omega_{ij}^\alpha)_{1 \leq i,j \leq m}^{1 \leq \alpha \leq n} \in L^n (\R^{m\times m}\otimes \R^n)$ is antisymmetric, i.e. $\Omega_{ij}^\alpha = -\Omega_{ji}^\alpha$ for all $1 \leq i,j \leq m$. Is it true that $u$ is continuous?
 \end{problem}          

\section{Partial evidence for regularity}

\subsection{Results on $n$-harmonic maps into symmetric targets}  

Shortly after the appearance of H\'{e}lein \cite{Helein-1991}, Bethuel~\cite{Bethuel-1992} and Evans~\cite{Evans-1991}, several authors have noted that the following result holds true.  

\begin{theorem}\label{roundsphere} Assume that $\nnn\subset \R^d$ is a round sphere $\Sphere^{d-1}$, or, more generally, a compact homogeneous space with a left-invariant metric. 
Then, all $n$-harmonic maps $u\in W^{1,n}(B^n,\nnn)$ are locally of class $C^\beta$ for some $\beta>0$.
\end{theorem}     

To the best of our knowledge, for $\nnn$ being a round sphere, the theorem was independently stated and proved by M.~Fuchs \cite{Fuchs-1993}, L.~Mou and P.~Yang \cite{MouY-1996b}, H.~Takeuchi \cite{Takeuchi-1994}, and the second named author in \cite{Strzelecki-1994}. The more general version for compact homogeneous spaces is due to T.~Toro and C.Y.~Wang \cite{ToroW-1995}.

There are, basically, three proofs of that result. Two rely on the duality of $BMO$ and the Hardy space, combined with the observation that for symmetric targets $\nnn$ the right-hand side of \eqref{phar} belongs to the 
(local) Hardy space. One of these two proofs is modelled on Evans' indirect blow-up argument, the other one employs a simple hole-filling trick. The third one \cite{Schikorra-2015} interprets the equation as a nonlocal system in the spirit of fractional harmonic maps, cf. \cite{DaLio-Riviere-2011,DaLio-Riviere-2011b,Schikorra-2012n2,DaLio-Schikorra-2014}.
         
Let us describe the argument in rather general terms, with emphasis on how the right--hand side is estimated.

\subsubsection{A possible approach to regularity} 
Let $u \in W^{1,n}$ be a solution to the system
\[
\dv(|\nabla u|^{n-2}\nabla u) =f(u,\nabla u),
\]
with $|f(u,\nabla u)\aleq |\nabla u|^n$. Let $[\cdot]_X$ be a critical semi-norm. Critical means: for any $\sigma > 0$ an estimate of the form
\begin{equation}\label{choiceofnorm}
\sup_{B_r} r^{-\sigma}[u]_{X,B_r} < \infty\, ,
\end{equation}
where the supremum is taken over balls $B_r$ with radius $r$, implies that $u$ is H\"older continuous.

Typical choices of such critical seminorms include
\[
[u]_{X,B_r} := \begin{cases} \|\nabla u\|_{L^n(B_r)},\\[4pt]
\displaystyle\sup_{B_\varrho(z)\subset B_r} \biggl( \varrho^{p-n}
\int_{B_\varrho(z)}
|\nabla u |^p\biggr)^{1/p} \qquad \mbox{for some $p < n$},\\[15pt]
\|\nabla u\|_{L^{(n,\infty)}(B_r)},\\[4pt] 
[u]_{BMO,B_r}.
\end{cases}
\]
(For definitions of the Lorentz spaces $L^{(p,q)}$ and the space $BMO$ of functions having bounded mean oscillation we refer to L.~Grafakos' monographs \cite{Grafakos-2014a,Grafakos-2014b}).

In all instances that we are aware of, one works with a semi-norm that satisfies
\begin{equation}\label{dominancebyLn}
[u]_{BMO, B_r} \aleq [u]_{X,B_r} \aleq \|\nabla u\|_{L^n(B_r)}\, .
\end{equation}
The goal is usually to obtain a decay estimate of the form
\begin{equation}
\label{goal}
[u]_{X,B_r} \leq (\tilde{\tau} + C\|\nabla u\|_{L^n(B_{2r})}^\gamma) [u]_{X,B_{2r}},
\end{equation}
with some fixed constants $\tilde\tau\in (0,1)$, $C$ and $\gamma>0$, independent of the specific ball $B_r$. It is good to think that $\tilde\tau$ is responsible simply for localization and cutting off the solution, whereas $C$ and $\gamma$ are related to the \emph{structure} of the right-hand side, $f(u,\nabla u)$. Here, the right choice of the semi-norm becomes extremely important; we shall come to that point later on.

On all balls $B_{2r}$ with $r$ small enough, due to the absolute continuity of the integral, we can estimate \eqref{goal} further to obtain
\[
[u]_{X,B_r} \leq \tau [u]_{X,B_{2r}}, \qquad \mbox{for some }\tau \in (\tilde\tau, 1).
\]
Once such estimate is ensured, we can employ an iteration scheme on smaller and smaller balls, see \cite[Chapter III]{giaquinta-book}, to obtain
\[
[u]_{X,B_r} \aleq r^\sigma\ [u]_{X,\R^n}.
\]
According to \eqref{choiceofnorm}, this then implies H\"older continuity.

So how to obtain an estimate of type \eqref{goal}? There are two ingredients needed: Firstly\footnote{A word of caution: one has to be careful with the boundary data, localize the solution properly etc., but these technical difficulties are minor or at best moderate; here and in the sequel, we have decided to sweep them under the rug in order to show only the essence of the arguments.}, for an estimate of the left-hand side  we need to find a seminorm $[\cdot]_Y$ so that
\begin{equation}\label{lhsest}
[u]_X^{n-1} \aleq \sup_{[\varphi]_Y \leq 1} \int |\nabla u|^{n-2}\nabla u\cdot \nabla \varphi.
\end{equation}
And then we need to obtain the right-hand side estimate of the form
\begin{equation}\label{rhsest}
 \int f \varphi \aleq [u]_X^{s} [\varphi]_Y, \qquad\mbox{where $s>n-1$}\, .
\end{equation}
(Typically, one strives for $s=n$, reflecting the degree of homogeneity of the right hand side in $u$; values of $s\le n-1$ do not yield any significant gain). 

One of the sources of the difficulties is that the canonical choice $X = Y = W^{1,n}(\R^n)$ contains \emph{unbounded} functions. This is why one is often forced to look for `non-standard' function spaces $X$ and $Y$ in the scheme described above. For the same reason, the \emph{structure} of $f$ matters a lot.

\subsubsection{Example: How this approach works in the sphere case} 
To give an idea how this works in practise, let us regard the case of the second author's proof of Theorem~\ref{roundsphere},  \cite{Strzelecki-1994}:

The $n$-harmonic map system is
\begin{equation}
\label{phar-s}
-\dv (|\nabla u|^{n-2}\nabla u) = u |\nabla u|^n.
\end{equation}
Now, 
\[
u^i|\nabla u|^{n} = \left ( u^i |\nabla u|^{n-2}\nabla u^j -  u^j |\nabla u|^{n-2}\nabla u^i \right ) \cdot \nabla u^j.
\]
This
is due to the fact that $u^j \nabla u^j = \frac{1}{2} \nabla |u|^2 = 0$ since $u$ belongs to the sphere -- we remind the reader that we use Einstein's summation convention. This formulation is useful, since \eqref{phar-s} implies
\[
\dv \left ( u^i |\nabla u|^{n-2}\nabla u^j -  u^j |\nabla u|^{n-2}\nabla u^i \right ) = 0\, ;
\]
that is, the initial system can be equivalently written as
\begin{equation}\label{eq:spherepde}
-\dv (|\nabla u|^{n-2}\nabla u) = \Omega \cdot \nabla u.
\end{equation}
where $\dv \Omega = 0$. Then, by the div-curl lemma \cite{CLMS}, the term $\Omega \cdot \nabla u$ belongs to the Hardy space $\mathcal{H}^1$ and the duality of Hardy space and BMO implies that
\[
\int  \varphi \; (\Omega \cdot \nabla u)  \; \aleq \; [\varphi]_{BMO}\ \|\Omega\|_{L^{n'}}\ \|\nabla u\|_{L^n} \; \aleq \; \|\nabla u\|_{L^n}^{n}\ \|\nabla \varphi\|_{L^n}.
\]
So taking $[\varphi]_X = [\varphi]_Y := \|\nabla \varphi\|_{L^n}$, 
we obtain a right-hand side estimate as in \eqref{rhsest},
with $s = n$.
The corresponding left-hand side estimate as in \eqref{lhsest} is easy:

\[
\|\nabla u\|_{L^n}^{n-1} \leq \sup_{\|\nabla \varphi \|_{L^n}\leq 1} \int |\nabla u|^{n-2}\nabla u\cdot \nabla \varphi.
\]
All this yields \eqref{goal} with $\gamma=1$, and allows to conclude that $u\in C^\beta$ for some $\beta>0$. 

\subsubsection{Why this is problematic in the general case} 
The main problem in the general case is that the right-hand side potential $\Omega$ in \eqref{eq:spherepde} may not be divergence free. For the $H$-system \eqref{hsystem} it is divergence free up to the multiplication of $H(u)$, a term which belongs to $L^\infty \cap W^{1,n}$. 
As it is done in \cite{Riviere-2007} for Rivi\`{e}re's general system \eqref{rivieresystem} for $n=2$, one might try to adopt Uhlenbeck's gauge transform for the case $n>2$ and \eqref{nrivieresystem}, see \cite{Uhlenbeck-1982,Wang-2005, Riviere-2007,Schikorra-2010, Miskiewicz-2016}. This could \emph{essentially} reduce \eqref{nrivieresystem} to a system with a potential $\tilde{\Omega} \in L^{n'}$ which is divergence free up to a multiplicative term which belongs to $L^\infty \cap W^{1,n}$. 

However, while this is sufficient for $n=2$ and \eqref{rivieresystem}, for $n > 2$ and \eqref{nrivieresystem} this leads to problems: for example, if we choose to work with Lorentz spaces $L^{p,q}$, and consider the right-hand side of the $H$-system. A slight extension of the div-curl estimate \cite{CLMS} coupled with Hardy-BMO-inequality gives that
\begin{equation}\label{eq:slightextensionHardyBMO}
\int H(u) \Jac u \varphi \aleq \|\nabla u\|_{L^{(n,q_1)}} \ldots \|\nabla u\|_{L^{(n,q_n)}}\ \|\nabla (H(u) \varphi)\|_{L^{(n,q_{n+1})}},
\end{equation}
where $q_i \in [1,\infty]$ and $\sum_{i=1}^{n+1} \frac{1}{q_i} = 1$. The first observation is that the test-function $\varphi$ should belong to $L^\infty$ to make the right-hand side bounded. By Sobolev embedding we should then choose in \eqref{lhsest} $Y = L^{(n,1)}$, and thus $X = L^{(n,\infty)}$. Hence, in order to at least formally match \eqref{lhsest} and \eqref{rhsest} to an estimate of the form \eqref{goal}, we need the $L^{(n,\infty)}$-norm of $\nabla u$ 
to appear at least $(n-1)$ times in \eqref{eq:slightextensionHardyBMO}. That is we should take $q_1,\ldots,q_{n-1} = \infty$. But since $\nabla H(u) \in L^n = L^{(n,n)}$, $q_{n+1}$ can be at most $n$. Thus $q_{n}$ \emph{has to be} at least $n'$. And unless $n = n' = 2$, we have $L^{(n,n')} \subsetneq L^{n}$, so that the natural assumption $u\in W^{1,n}$ does not suffice to close the argument.

That is, this kind of numerology magically fits and leads to the desired result  when $n =2$, i.e. when working with harmonic maps. But for $n > 2$ the exponents simply do not add up. This happens for all the choices for $X$ and $Y$ that we are aware of.

Let us also remark that even the additional assumption $\nabla u \in L^{(n,n')}$, would not solve this dilemma. The left-hand side estimate \eqref{lhsest} is to our best knowledge unknown for $X = L^{(n,\infty)}$ and $Y = L^{(n,1)}$. A reformulation of this additional problem is the following.

\begin{problem}\label{rieszest}
Let $\Rz = (\Rz_1,\ldots,\Rz_n)$ be the vectorial Riesz transform. Is it true or false that
\[
\|f\|_{L^{(p,\infty)}}^{p-1} \leq C\ \|\sum_{\alpha = 1}^n \Rz_{\alpha} (|\Rz f|^{p-2} \Rz_\alpha f)\|_{L^{(p',\infty)}} \, ?
\]
\end{problem}
Note that this problem can be easily solved for $p = 2$: for some constant $c \in \R$ we have $\sum_{\alpha=1}^n \Rz_\alpha \Rz_\alpha = c\, id$ and we can argue simply by $L^{2,1}$-$L^{2,\infty}$ duality. If instead of $L^{(p,\infty)}$ and $L^{(p',\infty)}$ one considers $L^{p}$ and $L^{p'}$, respectively, then it follows from the same simple duality argument. Also, if one replaces $L^{(p,\infty)}$ and $L^{(p',\infty)}$ by $L^{(r,q)}$ and $L^{(\frac{r}{p-1},\frac{q}{p-1})}$, respectively, then the analogue estimate holds whenever $|p-r|$ is small but non-zero. This uses  Iwaniec' stability theorem \cite{Iwaniec-1992} which follows from a nonlinear commutator estimate (see his monograph \cite{Iwaniec-Martin-2001}, and for an adaptation of the idea to another context \cite{Schikorra-2016}). 

Such a commutator estimate can be interpolated, but since it is nonlinear the interpolation can only take place between the space where the operator is Lipschitz (which happens exactly for $L^p$) and where the commutator is bounded. Thus one cannot obtain any interpolation estimates for the $L^{(p,q)}$-Lorentz spaces: they cannot be represented as interpolation spaces between $L^p$ and some $L^r$ (the best one can try is to work with the \emph{grand} $L^{p)}$-space, cf. \cite{Greco-Iwaniec-Sbordone-1997}).

A remark in passing: all the attempts at a proof in the 
general case try
to save the main idea of a hole--filling trick. 
Is any other approach possible?

\subsection{$H$-systems: results that use additional hypotheses on $H$ or $u$} 
For $n=2$, $H: \R^3 \to \R$, and $u \in W^{1,2}(\R^2,\R^3)$ the $H$-system is
\begin{equation}\label{eq:classichsystem}
\lap u = 2H(u)\, u_x \times u_y. 
\end{equation}
It was probably the first time that compensation effects for these kind of systems was observed, when H. Wente \cite{Wente-1969} proved what later should become to be known as Wente's inequality, see \cite{Brezis-Coron-1984,Tartar-1985}: Whenever $H$ is constant, any solution to \eqref{eq:classichsystem} is continuous.
F.~Tomi \cite{Tomi-1973} proved that bounded solutions $u$ to \eqref{eq:classichsystem} are $C^{1,\alpha}$ when $H$ is Lipschitz. 
E.~Heinz \cite{Heinz-1975} considered unbounded solutions $u$ for $H$ being Lipschitz and satisfying an extra decay-at-infinity condition.

Later on, M.~Gr\"uter \cite{Grueter-1981} proved regularity for all solutions to \eqref{eq:classichsystem} under the condition that $u$ is conformal and $H$ is bounded. Dropping the conformal parametrization, following the work on harmonic maps by H\`elein \cite{Helein-1991}, Bethuel \cite{Bethuel-1992} was able to prove regularity for possibly unbounded solutions whenever $H$ is bounded and Lipschitz.

Finally, Rivi\`{e}re \cite{Riviere-2007} proved that any $W^{1,2}$-solution $u$ to \eqref{eq:classichsystem} is Lipschitz, whenever $H$ is just bounded -- which proved a conjecture by Heinz.
For further properties and an overview we refer to \cite{DuzaarS-1996}.

For $n \geq 2$, much less is known for $W^{1,n}$-solutions $u$ to the $n$-dimensional $H$-system \eqref{hsystem}. Mou and Yang~\cite{MouY-1996a} showed continuity of \emph{conformal} solutions $u$ for $H$ is bounded. For $H$ Lipschitz and decaying at infinity,\footnote{Basically, one chooses the decay speed so that $H(u)u$ is of class $W^{1,n}\subset BMO$; this allows to cope with the fact that $BMO$ itself is not an algebra.}
Wang \cite{Wang-1999} proved that solutions of \eqref{hsystem} are continuous. For $H=\text{const}$ all solutions of class $W^{1,n}$ are regular; this is proved by the same argument as Theorem~\ref{roundsphere}, see e.g. Mou and Yang \cite{MouY-1996b}.

When $H$ is Lipschitz and $u$ not necessarily conformally parametrized, the following is known.

\begin{theorem}[Continuity under additional assumptions on $u$]\label{addassumptions} 
Let $$u=(u^1,u^2,\ldots,u^{n+1})\in W^{1,n}(\Omega,\R^{n+1})$$ be a weak solution of the $H$-system \eqref{hsystem}, where $H\colon \R^{n+1}\to \R$ is bounded and Lipschitz. If moreover one of the following conditions holds
\begin{enumerate}
\item $u \in L^\infty$
\item $u \in W^{n-1,n'}$
\item $\nabla u \in L^n \log^{n-1-\varepsilon} L$ for a certain $\eps > 0$
\end{enumerate}
then $u$ is continuous.
\end{theorem}
The first statement, regularity of \emph{bounded} weak solutions for $n>2$, is due to
Duzaar-Fuchs \cite{DuzaarF-1991a}. We sketch a proof of a slightly generalized version below, see Proposition~\ref{uboundedHcalpha}. The second statement is due to S. Kolasi\'nski, \cite{Kolasinski-2010} who combined 
Hodge decomposition, 
the Iwaniec stability theorem \cite{Iwaniec-1992}, and sharp Gagliardo-Nirenberg estimates due to the second author \cite{Strzelecki-2006}. We describe a slight improvement on Kolasi\'nski's argument below in Proposition~\ref{w22}. The third statement in Theorem~\ref{addassumptions} is due to the first-named author \cite{Schikorra-2013}. Instead of using the Gagliardo-Nirenberg-estimate, exponential Orlicz-spaces and Trudinger's inequality are used. 

Observe that for $n=2$, the second condition
reduces to the usual assumptions. All of the conditions above do not trivialize the problem: one can easily construct systems with the same growth properties that have solutions with singularities of the type $\log \log 1/|x|$.

If, as in the first statement above, all coordinate functions are bounded, one can essentially test the equation with $u$ itself (we use this for the proof of Proposition~\ref{uboundedHcalpha} below). This is not possible any more if $u$ is unbounded: since the right-hand side is merely integrable, all test functions have to be in $W^{1,n} \cap L^\infty$. This is also related to the fact that $W^{1,n}$ is not an algebra, but $W^{1,n} \cap L^\infty$ is. 

A remark in passing: due to the Jacobian term on the right-hand side of \eqref{hsystem} one might hope that it suffices to assume that \emph{only some} coordinate functions are bounded,  integrating by parts on the right-hand side. However the following is still open:
\begin{problem} Let $u=(u^1,u^2,\ldots,u^{n+1})\in W^{1,n}(\Omega,\R^{n+1})$ be a weak solution of the $H$-system \eqref{hsystem}, where $H\colon \R^{n+1}\to \R$ is bounded and Lipschitz. Assume that $k$ of the coordinate functions $u^i$ are bounded, where $k\in \{1,2,\ldots,n\}$. Is $u$ continuous?
\end{problem}

\subsection{How could a counterexample look like?}  
Duzaar and Fuchs \cite{Duzaar-Fuchs-1990} have proved that isolated singular points of $n$-harmonic maps are removable. Therefore, a possible counterexample to regularity would have to be singular on a perfect set; it is hard to imagine how a construction of such an example might look like when no nesting of isolated singular points is possible. 
                                                           
On the other hand, C. Wang \cite{Wang-2005}, see also \cite{Miskiewicz-2016}, and the second named author with A. Zatorska-Goldstein \cite{StrzeleckiZG-2004} have proved that for  \eqref{phar} and \eqref{hsystem} -- and for the corresponding approximate problems with a small perturbation in $(W^{1,n})^\ast$ added to the right hand side -- the spaces of weak solutions are closed in the weak topology of $W^{1,n}$. We do not know whether these results could be used to construct counterexamples from sequences of singular solutions to perturbed versions of \eqref{phar} and \eqref{hsystem}.

\subsection{New observations}

If $u$ is a bounded solution to the $H$-system, actually any H\"older continuous $H$ suffices to obtain continuity (which is a slight improvement from Theorem~\ref{addassumptions} where Lipschitz-continuity is required).
\begin{fact}\label{uboundedHcalpha}
Assume that $u\in W^{1,n}(B^n,\R^{n+1})$ is a bounded weak solution 
of \eqref{hsystem}, where $H\colon \R^{n+1}\to\R$ is 
a bounded function of class  $C^\alpha$ for some 
$\alpha\in (0,1)$. Then, $u\in C^\beta_{\text{loc}}$ for some $\beta >0$.	
\end{fact}            

The proof is based on the fact that all $BMO_p$--norms 
are equivalent for $1<p<\infty$. Thus for any $\alpha \in (0,1]$,
\[
[H(u)]_{BMO} \aleq [H]_{C^{\alpha}} [u]^\alpha_{BMO}.
\]
Since moreover
\[
[fg]_{BMO} \aleq \|f\|_{L^\infty} [g]_{BMO} + \|g\|_{L^\infty} [f]_{BMO},
\]
we can estimate
\[
[H(u)\varphi]_{BMO} \aleq \|\varphi\|_\infty [H]_{C^{0,\alpha}} [u]^\alpha_{BMO} + \|H\|_{L^\infty}\ [\varphi]_{BMO}.
\]
With the equation \eqref{hsystem} we then have the estimate
\[
\int |\nabla u|^{n-2}\nabla u \nabla \varphi \aleq \|\nabla u\|_{n}^n \left (\|\varphi\|_\infty [H]_{C^{0,\alpha}} [u]^\alpha_{BMO} + \|H\|_{L^\infty}\ [\varphi]_{BMO} \right ).
\]
Since $u$ is bounded, we 
we can proceed as explained above for the equation \eqref{phar-s}. Just pick $\varphi$ to be a cut-off version of $u$ to obtain \eqref{goal}.\qed

\begin{fact} \label{w22} If $n=4$ and $u\in W^{2,2}(B^4,\R^5)$ is a weak solution of \eqref{hsystem}, with $H$ Lipschitz and bounded, then,  $u\in C^\beta_{\text{loc}}$ for some $\beta >0$. 
\end{fact}              

Again, let us explain the argument briefly; an interested reader will easily fill in all the missing technical details. The three essential tools are: the Hodge decomposition (including a stability theorem due to Iwaniec); the Coifman--Rochberg--Weiss commutator theorem~\cite{CRW}, and the Gagliardo--Nirenberg inequalities in a sharp form, involving $BMO$ norms instead of $L^\infty$ \cite{Strzelecki-2006}.

Fix a small parameter $\eps>0$. Assuming w.l.o.g. that $u$ is compactly supported in $B_{2r}$, we construct the test function $\varphi$ via the Hodge decomposition,
\[
|\nabla u|^{-\eps}\nabla u = \nabla \varphi + V,
\]
where the divergence-free term $V$ is small due to the Stability Theorem \cite{Iwaniec-1992},
\begin{equation}
\label{eps-stability}
\| V\|_{L^{(n-\eps)/(1-\eps)}} \aleq \eps \|\nabla u\|_{L^{n-\eps}}\, .
\end{equation}
The main point is that $(n-\eps)/(1-\eps)>n$, and thus by Sobolev imbedding 
\begin{equation}
\text{osc}\, \varphi \aleq r^{(n-1)\eps/(n-\eps)} \biggl(\int_{B_{2r}}|\nabla u|^{n-\eps}\biggr)^{\frac{1-\eps}{n-\eps}}.
\end{equation}

After a routine argument (let us assume here for the sake of simplicity that also $\varphi$ is compactly supported), the left-hand side of the equation, due to the stability estimate \eqref{eps-stability}, gives
\begin{equation}
\label{lhs-Hodge}
\int |\nabla u|^{n-2}\nabla u \nabla \varphi \ge \int_{B_r} |\nabla u|^{n-\eps} - \tau \int_{B_{2r}} |\nabla u|^{n-\eps}
\end{equation}
with $\tau = \tau (n,\eps)\in (0,1)$.

To estimate the right hand side, we use the Hodge decomposition again, this time in $L^n$. It is convenient here to use the language of differential forms. We write the $H$-system as \eqref{H-forms} and work separately with each of the equations of the system. W.l.o.g. let $i=n+1$; write
\[
H(u)\, du^1=d\alpha + \delta\beta\, .
\]
Let $T$ denote the linear operator which maps a 1-form to the exact component of its Hodge decomposition; i.e. $T\big(H(u)du^1\big)=d\alpha$. Then,
the coexact term 
\[
\delta\beta= H(u)T(du^1)-T\big(H(u)du^1\big)=[H,T](du^1),
\]where $[H,T]$ is the commutator of $T$ and the multiplication by $H(u)$. Since $H$ is Lipschitz, by the Coifman-Rochberg-Weiss commutator theorem~\cite{CRW}
\[
\|\delta\beta\|_{L^n}\aleq [u]_{BMO} \|\nabla u\|_{L^n}\, .
\]
Now, we split the right-hand side
as
\[
\int \varphi H(u) \, du^1\wedge \ldots \wedge du^4 = \int \varphi\, d\alpha \wedge du^2\wedge du^3 \wedge
du^4
+\int \varphi\, \delta\beta \wedge du^2\wedge du^3 \wedge du^4\, .
\]
The first term, after one integration by parts -- taking $d$ from one of the $u^j$ to $\varphi$ -- and an application of the duality of the Hardy space and $BMO$, is controlled by a constant multiple of
\[
[u]_{BMO} \|\nabla \varphi\|_{L^{(4-\eps)/(1-\eps)}} \|\nabla u\|^3_{L^{4-\eps}} \aleq [u]_{BMO,B_{2r}} 
\int_{B_{2r}} |\nabla u|^{4-\eps}\, .
\]
The assumptions $n=4$ and $u\in W^{2,2}$ are crucial in the estimate of the second term, the one containing $\delta\beta$. We have
\begin{multline*}
\biggl|\int \varphi\, \delta\beta \wedge du^2\wedge du^3 \wedge du^4\biggr| 
 \le  \|\varphi \|_{L^\infty} \|\delta\beta\|_{L^4}\|\nabla u\|_{L^4}^3\\
 \aleq  r^{3\eps/(4-\eps)} \biggl(\int_{B_{2r}}|\nabla u|^{4-\eps}\biggr)^{\frac{1-\eps}{4-\eps}}
[u]_{BMO,B_{2r}} \int_{B_{2r}} |\nabla u|^4+\text{L.O.T.}\\
 \aleq  r^{3\eps/(4-\eps)} \biggl(\int_{B_{2r}}|\nabla u|^{4-\eps}\biggr)^{\frac{1-\eps}{4-\eps}}
[u]_{BMO,B_{2r}}^3\int_{B_{2r}} |D^2 u|^2 +\text{L.O.T.}\, 
\end{multline*}
In the last step, we used the sharp Gagliardo-Nirenberg inequality \cite{Strzelecki-2006}; L.O.T. stands for unimportant lower order terms.
Now, by Poincar\'{e} inequality, we have
\[
[u]_{BMO,B_{2r}} \aleq M(p,2r):=\sup_{B(z,\varrho)\subset B_{2r}} \biggl(\frac{1}{\varrho^{n-p}}\int_{B(z,\varrho)} |\nabla u|^p\biggr)^{1/p}\, .
\]
Thus, for $p=n-\eps=4-\eps$ the estimate of the second term can be written as
\begin{equation}
\label{2term}
\biggl|\int \varphi\, \delta\beta \wedge du^2\wedge du^3 \wedge du^4\biggr| \aleq r^\eps \int |D^2u|^2\, \cdot\, M(4-\eps,2r)^{4-\eps}\, .
\end{equation}
Combining this with the estimate of the first term of the right side, we finally obtain
\begin{equation}
\biggl|\int \varphi H(u) \, du^1\wedge \ldots \wedge du^4\biggr| \aleq r^\eps \biggl(
[u]_{BMO,B_{2r}} +
\int_{B_{2r}} |D^2 u|^2\,\biggr) M(4-\eps,2r)^{4-\eps}\, .
\end{equation}
Going back to \eqref{lhs-Hodge}, dividing both sides by $r^\eps$, and using the monotonicity of $M(p,2r)$ 
as a function of the ball, one is now able to use a standard iterative argument and prove that
\[
\frac{1}{r^\eps}\int_{B_r} |\nabla u|^{4-\eps} \aleq r^\sigma \qquad\mbox{for a fixed $\sigma>0$ and all radii $r<r_0$.}
\]
By the Dirichlet Growth Theorem, $u$ must be H\"{o}lder continuous.

\begin{remark}
One can check that the above argument, generalized to even dimensions $n=2k$ bigger then 4, yields the following:

\smallskip

\emph{If $n=2k$ is even and $u\in W^{n/2,2}(B^n,\R^{n+1})$ is a weak solution of \eqref{hsystem}, with $H$ Lipschitz and bounded, then,  $u\in C^\beta_{\text{loc}}$ for some $\beta >0$. }
\end{remark}

\begin{remark}
It is easy to see that the argument used to prove Proposition~\ref{w22} breaks down if we assume that $H$ is only $C^{0,\alpha}$ instead of Lipschitz: if $H\in C^{0,\alpha}$, then the exponent of $M(4-\eps,2r)$ in  \eqref{2term} is in fact $2+(1-\eps)+\alpha$, so that for $\alpha <1$ the estimates of both sides simply do not match.
\end{remark}

\section{Approaches based on non-local problems and commutator estimates}

\subsection{Questions concerning commutators.} One of the methods that allow to bypass the duality of Hardy space and BMO in the proofs of regularity of solutions to similar problems involves, roughly speaking, representing the test function as a Riesz potential of its gradient. The operator
\[
If(x)=\int_{\R^n} K(x-y) f(y)\, dy
\]
where $K(x)=c_n x/|x|^n$ is the gradient of the fundamental solution of the Laplacian, is bounded on the Hardy space, $I\colon{\mathcal{H}^1}(\R^n)\to L^{n/(n-1)}(\R^n)$, see \cite{Fefferman-Stein-1972,Semmes-1994}. Variants of that approach have been successfully applied in \cite{Hajlasz-Strzelecki-1998} or \cite{Strzelecki-2003c}.

However, when dealing e.g. with \eqref{hsystem} and using a compactly supported test function $\varphi = I(\nabla \varphi)$, one immediately arrives at the expressions of the form
\begin{multline*}
\biggl|
\int \varphi H(u)\, du^1\wedge\ldots du^n 
\biggr|
 = 
\biggl|
\int \nabla\varphi\cdot I\bigl(H(u) \, du^1\wedge\ldots du^n\bigr) \biggr|\\
 \le 
\biggl|
\int H(u)\nabla\varphi\cdot I \bigl( du^1\wedge\ldots du^n\bigr) \biggr| + \biggl|
\int \nabla\varphi\cdot [H(u),I] \bigl( du^1\wedge\ldots du^n\bigr) \biggr|
\end{multline*}
The key term here is the commutator $[b,I]$ of $I$ and the multiplication by $b=H(u)$. It is known that such commutators are not bounded on the Hardy space, cf. \cite{Harboure-etal-1997}, but are bounded on certain subspaces of ${\mathcal{H}^1}(\R^n)$, cf. Perez \cite{Perez-1995} and  Ky \cite{Ky-1995}. Differentiation of $[b,I]f$ gives the term $\nabla b\cdot If$ plus a commutator $[b,S]f$, where $S$ is a singular integral operator. Thus, the following general problem -- which might be of independent interest to harmonic analysts -- is linked with the regularity questions that we consider here.

\begin{problem} Fix $H$  bounded and Lipschitz on $\R^{n+1}$, and $u \in W^{1,n}(\R^n,\R^{n+1})$, so that $b=H(u)\in W^{1,n}(\R^n) \cap L^\infty(\R^n)$.
\begin{enumerate}[(a)] 
\item What is the largest subspace $V\subset {\mathcal{H}^1}(\R^n)$ such that $[b,I]\colon V\to L^{n/(n-1)}(\R^n)$?  What is the answer if we replace $L^{n/(n-1)}(\R^n)$ by $W^{1,1}(\R^n)$? Does this $V$ contain the $n\times n$ minors of $Du$?
\item  What is the largest subspace $V\subset {\mathcal{H}^1}(\R^n)$ such that all commutators $[b,\Rz_\alpha]$ of the multiplication by $b$ with the Riesz transforms $\Rz_\alpha$, $\alpha=1,\ldots,n$, map $V\to L^{1}(\R^n)$? Does $V$ contain all the $n\times n$ minors of $Du$?
\item Does $du^1\wedge\ldots\wedge du^n$ have an atomic decomposition into $b$-atoms, i.e. standard atoms that are, in addition, orthogonal to $b$? (Cf. \cite{Perez-1995} for definitions.)
\item Do the answers to the above questions change if $u$ is not just an element of $W^{1,n}$, but also a solution to an $H$-system?
\end{enumerate}
\end{problem}
The characterizations in Perez \cite{Perez-1995} and  Ky \cite{Ky-1995} do not seem to be directly applicable here: formally, there are abstract definitions of $V$'s on which the commutators of fractional or singular integrals are bounded \emph{but} there is no clear, practical way of telling whether a given Jacobian belongs to such a space or not.

Again, a positive answer to questions (a)--(c) would pave the way to Problems~\ref{hsystempb} and~\ref{n-harmonic}.

\subsection{Systems with antisymmetric potentials.} For $n=2$ Rivi\`{e}re \cite{Riviere-2007} observed that Euler-Lagrange systems of conformally invariant variational functionals and in particular the $H$-system can be brought into the form of  \eqref{rivieresystem}. This allowed him to conclude regularity  for solutions if $H$ is bounded. 

This reformulation is also possible for $n \neq 2$ for harmonic maps and the $H$-system, see \cite[eq. III.23]{Riviere-2011}, where \eqref{rivieresystem} becomes \eqref{nrivieresystem}.
But even for toy cases (see below) it does not seem clear how solutions would regularize when $\Omega$ belongs merely to $L^n$.

For the $H$-system with Lipschitz $H$, one can split the right-hand side of the equation into a part which has a purely determinant structure and a part with an antisymmetric potential which is in the smaller Lorentz space $L^{(n,n/2)}\subsetneq L^n$.

\begin{fact}\label{RHSdecomposition}
Let $u \in W^{1,n} (\R^n,\R^{n+1})$ be a solution to the $H$-System \eqref{hsystem} for some $H$ satisfying $\|H\|_{\infty} + \|\nabla H\|_{\infty}  < \infty$. Then we can write 
\[
 \dv(|\nabla u|^{n-2}\nabla u^i) = A + B,
 \]
for $A$ and $B$ as follows: $A$ is a sum over determinants of the form
\[
\det\nolimits_{n \times n}(\nabla \omega^{i_1},\nabla u^{i_2},\ldots,\nabla u^{i_n}),
\]
for some $\omega \in W_{loc}^{1,n}(\R^n,\R^{n+1})$.
$B$ can be written  as 
\[
B = \Omega_{ij}\cdot |\nabla u|^{n-2} \nabla u^j,
\]
with an $\Omega \in L_{loc}^{(n,\frac{n}{2})}(\R^n,so(n+1)\otimes \R^n)$.
\end{fact}
\begin{proof}
Again, it will be beneficial to use the language of differential $\ell$-forms $\Ep^\ell \R^{n+1}$. Let $y^i$ be the coordinates on $\R^{n+1}$. An orthonormal basis for $\Ep^\ell \R^{n+1}$ is then
\begin{equation}\label{Ep-basis}
\left \{dy^{i_1} \wedge \ldots \wedge dy^{i_\ell}: \quad 1 \leq i_1 < \ldots < i_\ell \leq n+1 \right \}.
\end{equation}
The Hodge-star operator $\star: \Ep^\ell \R^{n+1} \to \Ep^{n+1-\ell} \R^{n+1}$ maps $\ell$-forms $\omega$ into $n+1-\ell$-forms $\star \omega$ so that
\[
 \omega \wedge \star \omega = |\omega|^2.
\]
Here $|\omega|$ is the norm on $\Ep^{\ell}\R^{n+1}$ induced by the orthonormal basis in \eqref{Ep-basis}.

If we denote $u^\ast$ as the pullback of differential forms via $u$, the $H$-system \eqref{hsystem} can be written
\begin{equation}
\label{H-forms}
 \dv(|\nabla u|^{n-2}\nabla u^i) = H(u)\ u^\ast (\star dy^i).
\end{equation}
Since
\[
 \star dy^i  = \star (dy^i \wedge dy^j \wedge dy^k) \wedge dy^j \wedge dy^k,
\]
we can write \eqref{hsystem} as
\[
 \dv(|\nabla u|^{n-2}\nabla u^i) =  \Gamma_{ik} \wedge du^k,
\]
for
\[
 \Gamma_{ik} := \frac{1}{2n}  u^\ast \star (dy^i \wedge dy^j \wedge dy^k) \wedge u^\ast(Hdy^j).
\]
$\Gamma$ is clearly antisymmetric, $\Gamma_{ik} = -\Gamma_{ki}$. By Hodge decomposition on $\R^{n}$, 
\[
 u^\ast(H dy^j) = d\omega^j + \delta \beta^j,
\]
where $\omega^j \in W^{1,n}_{loc}(\R^n)$ and $\beta^j \in W^{1,n}_{loc}(\R^n,\Ep^{2} \R^n)$ satisfy the equations
\[ 
 \lap \omega^j = \dv(H(u) \nabla u),
\]
and
\begin{equation}\label{eq:lapbetaj}
\lap \beta^j = dH(u) \wedge du^j.
\end{equation}
Elliptic estimates then give
\[
 \|\omega^j\|_{W^{1,n}} \aleq \|H\|_\infty \|\nabla u \|_{L^{n}}.
\]
Equation \eqref{eq:lapbetaj} has a Jacobian structure on the right-hand side and $H(u) \in W^{1,n}$. That is why the estimates for $\beta$ is a little bit better,
\[
 \|\nabla \beta \|_{L^{(n,\frac{n}{2})}} \aleq \|\nabla u\|_{L^n}^2.
\]
We set
\[
 \Omega_{ik} := |\nabla u|^{2-n}\frac{1}{2n}\sum_{j=1}^n  u^\ast \star (dy^i \wedge dy^j \wedge dy^k) \wedge \delta\beta^j,
\]
which is still antisymmetric in $i$,$k$. Moreover, by the estimate above, $\Omega \in L^{(n,\frac{n}{2})}$.

On the other hand,
\[
u^\ast \star (dy^i \wedge dy^j \wedge dy^k) \wedge d\omega^j \wedge du^k
\]
is the $n$-form 
\[
 (-1)^{i+j+k} du^1 \wedge \ldots du^{i-1}\wedge du^{i+1} \wedge \ldots du^{j-1}\wedge du^{j+1}\ldots du^{k-1}\wedge du^{k+1} \wedge \ldots du^{n+1} \wedge d\omega^j \wedge du^k
 \]
which is the $(-1)^{i+j+k}$ times the volume element of $\R^n$ times the determinant
\[
 \det\nolimits_{n\times n} (\nabla u^1, \ldots, \nabla u^{i-1},\nabla u^{i+1},\ldots ,\nabla u^{j-1},\nabla u^{j+1},\ldots, \nabla u^{k-1},\nabla u^{k+1} , \ldots, \nabla u^{n+1}, \nabla \omega^j , \nabla u^k).
\]
\end{proof}

Let us now consider the right-hand side terms in Proposition~\ref{RHSdecomposition} separately. If we consider solutions to
\[
 \dv(|\nabla u|^{n-2}\nabla u^i) = A,
\]
where $A$ has the determinant structure as in Proposition~\ref{RHSdecomposition}, then continuity follows immediately via Hardy-BMO duality.

Solutions of the equation
\[
 \dv(|\nabla u|^{n-2}\nabla u^i) = B,
\]
where $B$ has the structure as in  Proposition~\ref{RHSdecomposition}, are still not understood, even if the the antisymmetric potential $\Omega$ satisfies a sharper Lorentz-space estimate. If $n=2$, the antisymmetric potential belongs to $L(2,1)$ and the equation regularizes without any use of the antisymmetry by standard potential estimates. If $n > 2$, it is unclear how the antisymmetry improves our chances. We thus propose a simplified version of Problem~\ref{rivierepb}.

\begin{problem}\label{Lorentzrivierepb}
Let $u \in W^{1,n}(B,\R^m)$ be a solution to \eqref{nrivieresystem} for 
where $\Omega$ is an $L^{(n,p)}$-matrix with values in $so(m)\otimes\R^n$. For which $p \in [1,n]$ is $u$ necessarily continuous?
\end{problem}  
If $p = 1$, i.e. $\Omega \in L(n,1)$, Lipschitz continuity of solutions follows -- even without antisymmetry of $\Omega$. This is due to 
Duzaar and Mingione \cite{Duzaar-Mingione}. We know of no argument which shows any improvement of regularity for solutions to \eqref{nrivieresystem} with antisymmetric potential $\Omega \in L^{(n,p)}$ for $p > 1$. Recall, that this is not only a problem of right-hand side estimates, but also the left-hand side estimates for this equation are unclear, see Problem~\ref{rieszest}.

One way to gain more insight in this situation is to derive somewhat simplified sub-problems of regularity for solutions to
\[
\dv(|\nabla u|^{n-2}\nabla u^i) = \Omega\cdot  |\nabla u|^{n-2}\nabla u^i.
\]
Let $\mathcal{R}_\alpha$ be the $\alpha$-th Riesz transform. For any vector field $F = (F_1,\ldots,F_n)$ we have a zero order Hodge decomposition
\[
F_\alpha = \Rz_\alpha \Rz_\beta F_\beta + \Rz_\beta (\Rz_\beta F_\alpha - \Rz_\alpha F_\beta). 
\]
Applying this to $|\nabla u|^{n-2} \partial_\alpha u$, we can decompose
\[
|\nabla u|^{n-2} \partial_\alpha u = \Rz_\alpha w + e_\alpha,
\]
where
\[
w :=\Rz_\beta (|\nabla u|^{n-2}\partial_\beta u). 
\]
Since $u\in W^{1,n}$ and the Riesz transforms $\Rz_\beta$ are bounded on $L^p$, $1<p<\infty$, the natural assumption here is $w\in L^{n'}$.
The above equation then changes into
\[
\laps{1} w^i = \Omega^\alpha_{ij}\Rz_\alpha [w^i] + \Omega^\alpha_{ij} e_\alpha. 
\]
Ignoring the error $e_\alpha$-part, this equation is of the form
\[
\laps{1} w^i = \Omega[w^i]
\]
where $\Omega$ is acting as a linear operator (and not just as a pointwise multiplication). In \cite{Schikorra-2012} it was shown that these potentials regularize the equation in $\R^2$ (a generalization of the arguments developed for pointwise multiplication operators \cite{DaLio-Riviere-2011b}, see also \cite{Riviere-2011Schroedinger}). This argument can be extended to obtain the following
\begin{theorem}
Let $w \in L^{n'}(\R^n)$ be a solution to
\[
\laps{1} w^i = \Omega^\alpha_{ij}\Rz_\alpha [w^i].
\]
If we assume $\Omega$ to be antisymmetric and $\Omega \in L^{(n,2)}$, then $w \in L^{p}$ for some $p > n'$.
\end{theorem}
The proof is very technical (because the equation is non-local). It follows closely the arguments in \cite{Schikorra-2012}, and we are not going to give it here. We rather give an argument for a related situation.

\begin{theorem}\label{th:lapuomegaun}
Let $u \in W^{1,n'}(B^n,\R^m)$ be a solution to
\begin{equation}\label{eq:lapuomegau}
-\lap u^i = \Omega^\alpha_{ij}\nabla u.
\end{equation}
If $\Omega$ is antisymmetric and $\Omega \in L^{(n,2)}$, then $\nabla u \in L^{p}_{loc}(B^n)$ for some $p > n'$.
\end{theorem}
\begin{proof}
Since this is a local result, we may assume that $\|\Omega\|_{L^{(n,2)}} \leq \eps$ for a small enough $\eps > 0$. Then, with Rivi\`{e}re's extension \cite{Riviere-2007} of Uhlenbeck's \cite{Uhlenbeck-1982}, see also \cite{Schikorra-2010}, we obtain a gauge $P\in W^{1,n}(B^n,SO(m))$ and $\|\nabla P\|_{L^{(n,2)}} \leq \|\Omega\|_{L^{(n,2)}}$ so that
\[
\dv (P^T \nabla P + P^T \Omega P) = 0.
\]
Plugging this into \eqref{eq:lapuomegau} we have
\begin{equation}\label{eq:uhlenbeckdec}
-\dv (P^T \nabla u) = (P^T \nabla P + P^T \Omega P)\, P^T \nabla u.
\end{equation}
We test the right-hand side with $\varphi \in C_c^\infty(B^n)$ and 
use a slightly improved $div$-$curl$-estimate, cf \eqref{eq:slightextensionHardyBMO}:
\[
 \int F \cdot \nabla g\ h \aleq \|F\|_{L^{(p,q_1)}}\ \|\nabla h\|_{L^{(p',q_2)}}\ \|\nabla g\|_{L^{(n,q_3)}},
\]
which holds whenever $p \in (1,\infty)$, $\frac{1}{q_1} + \frac{1}{q_2} + \frac{1}{q_3} = 1$, $1 \leq q_1,q_2,q_3 \leq \infty$ and $\dv (F) = 0$.
Then, 
\[
\begin{split}
\int (P^T \nabla P + & P^T \Omega P) \nabla (P^T \varphi) \, u\\
& \aleq\;  \|\nabla u\|_{L^{(n',\infty)}}\ \|P^T \nabla P + P^T \Omega P\|_{L^{(n,2)}}\ \| \nabla (\varphi P)\|_{L^{(n,2)}}\\
& \aleq\;  \|\nabla u\|_{L^{(n',\infty)}}\ \|\Omega\|_{L^{(n,2)}}\ (\|\varphi\|_{L^\infty} + \|\nabla \varphi\|_{L^{(n,2)}}).
\end{split}
\]
Using the Sobolev inequality
\[
\|\varphi \|_{L^\infty} \aleq \|\nabla \varphi\|_{L^{(n,1)}},
\]
and \eqref{eq:lapuomegau}, we obtain for any $\varphi \in C_c^\infty(B^n)$,
\[
 \biggl|\int P^T \nabla u\cdot \nabla \varphi \biggr| \;\aleq\; \eps\, \|\nabla u\|_{L^{(n',\infty)}}\, \|\nabla \varphi \|_{L^{(n,1)}}.
\]

Thus, using Hodge decomposition and suitably localizing, we obtain that for some $\tau < 1$ on any small balls $B_r$ it holds
\[
\|\nabla u\|_{L^{n',\infty}(B_r)} \leq \tau \|\nabla u\|_{L^{n',\infty}(B_{2r})}.
\]
An iteration argument now gives.
\[
\sup_{B_{2r} \subset B} r^{-\sigma}\|\nabla u\|_{L^{n',\infty}(B_r)} \aleq  \|\nabla u\|_{L^{n'}(B)}.
\]
Having this estimate, one repeats the above argument with $\laps{\eps} \varphi$ instead of $\varphi$, and obtains that
\[
\sup_{B_{2r} \subset B} r^{-\sigma}\|\laps{\eps}\nabla u\|_{L^{\frac{n}{n-1+\eps},\infty}(B_r)} < \infty
\]
Then Adams' estimates on Riesz potentials in Morrey spaces \cite{Adams-1975} give the claim.
\end{proof}

Let us remark that for $n=2$ a much more beautiful argument works for Theorem~\ref{th:lapuomegaun}. Under a smallness-assumption on $\|\Omega\|_{L^2}$, one can not only find $P$ as in \eqref{eq:uhlenbeckdec}, but Rivi\`{e}re \cite{Riviere-2007} was able to find $A \in W^{1,2}\cap L^\infty(B^2,GL(m))$ so that
\[
 \dv (\nabla A + A\Omega ) = 0.
\]
Then $u$ actually satisfies a conservation law, 
\[
 -\dv (A \nabla u - (\nabla A + A\Omega ) u) = 0,
\]
and regularity follows from a simple duality argument. The construction of $A$ however crucially depends on Wente's inequality \cite{Wente-1969,Tartar-1985,Brezis-Coron-1984}: if 
\[
 \lap f = G \cdot \nabla h \quad \mbox{in $\R^2$},
\]
and $\dv (G) = 0$, then
\[
 \|f\|_{L^\infty} \leq C\ \|G\|_{L^2}\ \|\nabla h\|_{L^2}.
\]
It is unknown how to construct such a map $A$ in $\R^n$ for $n > 2$: the counterpart of Wente's inequality for $n>2$ is missing.

Indeed, there are well known counterexamples that show $\Delta_n u =g \in \mathcal{H}^1 \not \Rightarrow u \in C^0$, see Firoozye \cite{Firoozye-1995}, and also Iwaniec and Onninen's counter-example for right-hand side in an Orlicz space \cite{Iwaniec-Onninen-2007}.

\subsection*{Acknowledgement.} The work of the first author has been partially supported by DFG-grant SCHI-1257-3-1. The work of the second author has been partially supported by the NCN grant no. 2012/07/B/ST1/03366.

\bibliography{hsys}{}
\bibliographystyle{amsplain}

\end{document}